\documentclass[leqno]{article}
\usepackage[frenchb,english]{babel}
\usepackage[utf8]{inputenc}
\usepackage{amsmath}
\usepackage{amssymb}
\usepackage{amsfonts}
\usepackage{enumerate}
\usepackage{vmargin}
\usepackage[all]{xy}
\usepackage{mathrsfs}
\usepackage{mathtools}
\usepackage{lmodern}
\usepackage{comment}%
\usepackage[colorlinks=true,pagebackref=true]{hyperref}

\setmarginsrb{3cm}{3cm}{3.5cm}{3cm}{0cm}{0cm}{1.5cm}{3cm}

\newcommand{\B}{\ensuremath{\mathrm{B}}}
\newcommand{\C}{\ensuremath{\mathbb{C}}}
\newcommand{\D}{\mathrm{D}}
\newcommand{\E}{\ensuremath{\mathbb{E}}}

\let\H\relax 
\newcommand{\H}{\mathrm{H}} 

\let\L\relax 
\newcommand{\L}{\mathrm{L}}

\newcommand{\R}{\ensuremath{\mathbb{R}}}

\newcommand{\W}{\mathrm{W}}

\newcommand{\e}{\mathrm{e}}

\renewcommand{\d}{\mathop{}\mathopen{}\mathrm{d}} 

\let\i\relax 
\newcommand{\i}{\mathrm{i}}
\newcommand{\Id}{\mathrm{Id}} 
\newcommand{\VN}{\mathrm{VN}} 

\renewcommand{\leq}{\ensuremath{\leqslant}}
\renewcommand{\geq}{\ensuremath{\geqslant}}
\newcommand{\qed}{\hfill \vrule height6pt  width6pt depth0pt}
\newcommand{\norm}[1]{ \| #1  \|}
\newcommand{\bnorm}[1]{ \big\| #1  \big\|}
\newcommand{\Bnorm}[1]{ \Big\| #1  \Big\|}

\newcommand{\xra}{\xrightarrow} 
\newcommand{\co}{\colon}

\newcommand{\ot}{\otimes}

\newcommand{\scr}{\mathscr} 

\newcommand{\ov}{\overset}

\DeclareMathOperator{\Aut}{Aut} 

\DeclareMathOperator{\dom}{dom}

\newtheorem{thm}{Theorem}[section]

\newtheorem{prop}[thm]{Proposition}

\newtheorem{lemma}[thm]{Lemma}

\newtheorem{remark}[thm]{Remark}

\newenvironment{proof}[1][]{\noindent {\it Proof #1} : }{\hbox{~}\qed
\smallskip
}

\let\OLDthebibliography\thebibliography
\renewcommand\thebibliography[1]{
  \OLDthebibliography{#1}
  \setlength{\parskip}{0pt}
  \setlength{\itemsep}{0pt plus 0.3ex}
}

\numberwithin{equation}{section}

\begin{document}
\selectlanguage{english}
\title{\bfseries{Dilations of markovian semigroups of Fourier multipliers on locally compact groups}}
\date{}
\author{\bfseries{C\'edric Arhancet}}

\maketitle

\begin{abstract}
We prove that any weak* continuous semigroup $(T_t)_{t \geq 0}$ of Markov Fourier multipliers acting on a group von Neumann algebra $\VN(G)$ associated to a locally compact group $G$ can be dilated by a weak* continuous group of Markov $*$-automorphisms on a bigger von Neumann algebra. Our construction relies on probabilistic tools and is even new for the group $\R^n$. Our results imply the boundedness of the McIntosh's $\H^\infty$ functional calculus of the generators of these semigroups on the associated noncommutative $\L^p$-spaces. 




\vspace{0.2cm}



%
%


%
%
%
%
%
%
%
%
%
%
\end{abstract}


\makeatletter
 \renewcommand{\@makefntext}[1]{#1}
 \makeatother

\section{Introduction}
\label{sec:Introduction}

The study of dilations of operators is of central importance in operator theory and has a long tradition in functional analysis. Indeed, dilations are powerful tools which allow to reduce general studies of operators to more tractable ones. 

Suppose $1<p<\infty$. In the spirit of Sz.-Nagy's dilation theorem for contractions on Hilbert spaces, Fendler \cite{Fen1} proved a dilation result for any strongly continuous semigroup $(T_t)_{t \geq 0}$ of positive contractions on an $\L^p$-space $\L^p(\Omega)$. More precisely, this theorem says that there exists a (bigger) measure space $\Omega'$, two positive contractions $J \co \L^p(\Omega) \to \L^p(\Omega')$ and $P \co \L^p(\Omega') \to \L^p(\Omega)$ and a strongly continuous group of positive invertible isometries $(U_t)_{t \in \R}$ on $\L^p(\Omega')$ such that 
\begin{equation}
\label{Fendler-dilation}
T_t = PU_tJ	
\end{equation}
for any $t \geq 0$, see also \cite{Fen2}. Note that in this situation, $J \co \L^p(\Omega) \to \L^p(\Omega')$ is an isometric embedding whereas $JP \co \L^p(\Omega') \to \L^p(\Omega')$ is a contractive projection.

In the noncommutative setting, measure spaces and $\L^p$-spaces are replaced by von Neumann algebras and noncommutative $\L^p$-spaces and positive maps by completely positive maps. In their remarkable paper \cite{JLM}, Junge and Le Merdy essentially\footnote{\thefootnote. The authors prove that there exists no ``reasonable'' analog of a variant of Fendler's result for a discrete semigroup $(T^k)_{k \geq 0}$ of completely positive contractions.} showed that there is no hope to have a ``reasonable'' analog of Fendler's result for semigroups of completely positive contractions acting on noncommutative $\L^p$-spaces. It is a \textit{striking} difference with the world of classical (=commutative) $\L^p$-spaces of measure spaces. 

The semigroups of selfadjoint Markov Fourier multipliers plays a fundamental role in noncommutative harmonic analysis and operator algebras, see e.g. \cite{Haa1}. 
In this paper, our main result (Theorem \ref{Th-dilation-continuous}) gives a dilation of weak* continuous semigroups of Markov Fourier multipliers acting on the group von Neumann algebra $\VN(G)$ of a locally compact group $G$ in the spirit of \eqref{Fendler-dilation} but at the level $p=\infty$. Our construction induces an isometric dilation similar to the one of Fendler's theorem for the strongly continuous semigroup induced by the semigroup $(T_t)_{t \geq 0}$ on the associated noncommutative $\L^p$-space $\L^p(\VN(G))$ for any $1 \leq p <\infty$. Note that our paper \cite{Arh4} gives a nonconstructive and complicated proof of such a dilation and effective only for \textit{discrete} groups. Here our approach is very different, direct and short. In addition, it can also be used with \textit{non-discrete} locally compact groups. We refer to \cite{Arh1}, \cite{Arh2}, \cite{Arh4}, \cite{ALM}, \cite{AFM}, \cite{HaM} and \cite{Ric} for strongly related things.

One of the important consequences of Fendler's theorem is the boundedness, for the generator of a strongly continuous semigroup $(T_t)_{t \geq 0}$ of positive contractions, of a bounded $\H^\infty$ functional calculus which is a fundamental tool in various areas: harmonic analysis of semigroups, multiplier theory, Kato's square root problem, maximal regularity in parabolic equations, control theory, etc. For detailed information, we refer the reader to \cite{Haa}, \cite{JMX}, \cite{KW}, to the survey  \cite{LeM1} and to the recent book \cite{HvNVW2} and references therein. Our theorem also gives a similar result on $\H^\infty$ functional calculus in the noncommutative context as explained in \cite[Proposition 3.12]{JMX} and \cite[Proposition 5.8]{JMX} in the case of semigroups of Fourier multipliers. 


The paper is organized as follows. The next Section \ref{sec:preliminaries} gives background. In particular, we give some information on crossed products since our construction relies on this notion. We also prove some elementary results which will be used in the sequel. Section \ref{sec:dilation-continuous} gives a proof of our main result of dilation of semigroups of Markov Fourier multipliers. 
In the last section \ref{sec:application}, we describe some applications of our results to functional calculus.

\section{Preliminaries}
\label{sec:preliminaries}

\paragraph{Isonormal processes} Let $H$ be a real Hilbert space. An $H$-isonormal process on a probability space $(\Omega,\mu)$ \cite[Definition 1.1.1]{Nua1} \cite[Definition 6.5]{Neer1} is a linear mapping $\W \co H \to \L^0(\Omega)$ with the following properties:
\begin{flalign}
& \label{isonormal-gaussian} \text{for any $h \in H$ the random variable $\W(h)$ is a centered real Gaussian,} \\
&\label{esperance-isonormal} \text{for any } h_1, h_2 \in H \text{ we have } \E\big(\W(h_1) \W(h_2)\big)= \langle h_1, h_2\rangle_H. \\
&\label{density-isonormal} \text{The linear span of the products } \W(h_1)\W(h_2)\cdots \W(h_m), 
\text{ with } m \geq 0 \text{ and } h_1,\ldots, h_m \\
&\nonumber\text{in }H, \text{ is dense in the real Hilbert space $\L^2_\R(\Omega)$.}\end{flalign}  
Here $\L^0(\Omega)$ denote the space of measurable functions on $\Omega$ and we make the convention that the empty product, corresponding to $m=0$ in \eqref{density-isonormal}, is the constant function $1$. 

If $(e_i)_{i \in I}$ is an orthonormal basis of $H$ and if $(\gamma_i)_{i \in I}$ is a family of independent standard Gaussian random variables on a probability space $\Omega$ then for any $h \in H$, the family $(\gamma_i \langle h,e_i\rangle)_{i \in I}$ is summable in $\L^2(\Omega)$ and
\begin{equation}
\label{Concrete-W}
\W(h)
\ov{\mathrm{def}}{=} \sum_{i \in I} \gamma_i \langle h,e_i\rangle_H, \quad h \in H
\end{equation}
define an $H$-isonormal process.

Recall that the span of elements $\e^{\i\W(h)}$ is weak* dense in $\L^\infty(\Omega)$ by \cite[Remark 2.15]{Jan1}.  
Using \cite[Proposition E.2.2]{HvNVW2}, we see that
\begin{equation}
\label{Esperance-exponentielle-complexe}
\mathrm{E}\big(\e^{\i t\W(h)}\big)
=\e^{-\frac{t^2}{2} \norm{h}_H^2}, \quad t \in \R, h \in H.
\end{equation}

If $u \co H \to H$ is a contraction, we denote by $\Gamma_\infty(u) \co \L^\infty(\Omega) \to \L^\infty(\Omega)$ the (symmetric) second quantization of $u$ acting on the \textit{complex} Banach space $\L^\infty(\Omega)$. Recall that the map $\Gamma_\infty(u) \co \L^\infty(\Omega) \to \L^\infty(\Omega)$ preserves the integral\footnote{\thefootnote. That means that for any $f \in \L^\infty(\Omega)$ we have $\int_{\Omega} \Gamma_\infty(u)f \d\mu=\int_{\Omega} f \d\mu$.}. If $u$ is an isometry we have
\begin{equation}
\label{SQ1}
\Gamma_\infty(u) \big(\e^{\i\W(h)}\big)
=\e^{\i\W(u(h))}, \quad h \in H
\end{equation}
and $\Gamma_\infty(u) \co \L^\infty(\Omega) \to \L^\infty(\Omega)$ is a $*$-automorphism of the von Neumann algebra $\L^\infty(\Omega)$.
Furthermore, the second quantization functor $\Gamma$ satisfies the following. In the part 1, we suppose that the construction\footnote{\thefootnote. The existence of a proof of Lemma \eqref{Lemma-semigroup-continuous} without \eqref{Concrete-W} is unclear.} is given by the concrete representation \eqref{Concrete-W}. We will only use this  observation in the non-discrete case.  

\begin{lemma}
\label{Lemma-semigroup-continuous} 
\begin{enumerate}
	\item If $\L^\infty(\Omega)$ is equipped with the weak* topology then the map $H \to \L^\infty(\Omega)$, $h \mapsto \e^{\i\W(h)}$ is continuous.
	\item If $\pi \co G \to \B(H)$ is a strongly continuous orthogonal representation of a  locally compact group, then $G \to \B(\L^\infty(\Omega))$, $s \mapsto \Gamma_\infty(\pi_s)$ is a weak* continuous\footnote{\thefootnote. That means that $\B(\L^\infty(\Omega))$ is equipped with the point weak* topology.} representation on the Banach space $\L^\infty(\Omega)$.
\end{enumerate}
\end{lemma}

\begin{proof}
1. Suppose that the sequence $(h_n)$ converges to $h$ in $H$. Note that there exists an at most countable subset $J$ of $I$ such that $i \in I-J$ implies $\langle h_n,e_i\rangle_H=0$ and $\langle h,e_i\rangle_H=0$. By \eqref{Concrete-W}, note that $\W(h_n)= \sum_{i \in J} \gamma_i \langle h_n,e_i\rangle_H$ and $\W(h)= \sum_{i \in J} \gamma_i \langle h,e_i\rangle_H$ in $\L^2(\Omega)$, hence almost everywhere by \cite[Corollary 6.4.4]{HvNVW2}. For any $i \in J$, the sequence $(\langle h_n,e_i\rangle_H)$ converges to $\langle h,e_i\rangle_H$. Since $J$ is at most countable, $\W(h_n)$ converges almost everywhere to $\W(h)$. It is not difficult to conclude using the dominated convergence theorem.


2. The second part can be proved with the arguments of the proof of \cite[Lemma 9.3]{JMX}.

\end{proof}

Let $H$ be a real Hilbert space. Following \cite[Definition 2.2]{NVW1} and \cite[Definition 6.11]{Neer1}, we say that an $\L^2_\R(\R^+,H)$-isonormal process $\W$ is an $H$-cylindrical Brownian motion. In this case, for any $t \geq 0$ and any $h \in H$, we let
\begin{equation}
\label{Def-Wt-h}
\W_t(h)\ov{\mathrm{def}}{=}\W\big(1_{[0,t]} \ot h\big).	
\end{equation}
We introduce the filtration $(\scr{F}_t)_{t \geq 0}$ generated by $\W$ defined by 
\begin{equation}
\label{Filtration-H-cylindrical}
\scr{F}_t\ov{\mathrm{def}}{=} \sigma\big(\W_r(h) : r \in [0,t], h \in H\big).	
\end{equation} 
By \cite[pages 77]{Neer1}, for any fixed $h \in H$, the family $(\W_t(h))_{t \geq 0}$ is a Brownian motion that is \cite[Definition 6.2]{Neer1}
\begin{flalign}
& \label{isonormal-almost} \text{$\W_0(h) = 0$ almost surely,} \\
&\label{difference-1} \text{$\W_t(h)-\W_s(h)$ is Gaussian with variance $t-s$ for any $0 \leq s \leq t$,} \\
&\label{difference-2} \text{$\W_t(h)-\W_s(h)$ is independent of $\{\W_r(h) :  r \in [0 ,s] \}$ for any $0 \leq s \leq t$}.  
\end{flalign}
Indeed by \cite[163]{Neer1}, 
\begin{flalign}
\label{difference-3}
&\text{the increment $\W_t(h)-\W_s(h)$ is independent of the $\sigma$-algebra $\scr{F}_s$}.
\end{flalign}
Moreover, by \cite[163]{Neer1} the family $(\W_t(h))_{t \geq 0}$ is a martingale with respect to $(\scr{F}_t)_{t \geq 0}$. In particular, the random variable $\W_t(h)$ is $\scr{F}_t$-measurable. If $0 \leq s \leq t$, note that 
$$
\norm{1_{]s,t]} \ot h}_{\L^2_\R(\R^+,H)}^2
= \bnorm{1_{]s,t]}}_{\L^2_\R(\R^+)}^2 \norm{h}_H^2
=(t-s) \norm{h}_H^2.
$$ 
Using \eqref{Esperance-exponentielle-complexe} together with the previous computation, we obtain
\begin{equation}
\label{Esperance-exponentielle-complexe-2}
\mathrm{E}\big(\e^{\i \W(1_{]s,t]} \ot h)}\big)
=\e^{-\frac{t-s}{2} \norm{h}_H^2}, \quad 0 \leq s \leq t,\ h \in H.
\end{equation}



\paragraph{Von Neumann algebras} 
Let $M$ be a von Neumann algebra. We denote by $\mathrm{U}(M)$ the group of all unitaries of $M$. Suppose that $M$ is equipped with a semifinite normal faithful weight $\psi$. We denote by $\mathfrak{n}_\psi$ the left ideal of all $x \in M$ such that $\psi(x^*x)<\infty$.



Let $M$ and $N$ be von Neumann algebras equipped with faithful normal semifinite weights $\psi_M$ and $\psi_N$. We say that a positive linear map $T \co M \to N$ is weight preserving if for any $x \in \mathfrak{m}_{\psi_M}^+$ we have $\psi_N(T(x))=\psi_M(x)$. 

Note that a locally convex space $X$ is said to be quasi-complete if every bounded Cauchy net of $X$ converges \cite[Definition 4.23]{Osb1}. Recall that if $H$ is a Hilbert space and if $M$ is a von Neumann algebra acting on $H$ then the $\sigma$-strong* topology on $M$ is defined by the seminorms $x \mapsto \big(\sum_{n=1}^{\infty} (\norm{x(\xi_n)}^2+\norm{x^*(\xi_n)}^2)\big)^{\frac{1}{2}}$ where $\xi_n \in H$ and $\sum_{n=1}^{\infty} \norm{\xi_n}^2<\infty$. The following lemma is folklore. Unable to locate this statement in the literature, we give a short argument.

\begin{lemma}
\label{Lemma-quasi-complete}
Let $M$ be a von Neumann algebra acting on a Hilbert space $H$. Then $M$ equipped with the canonical locally convex structure which gives the $\sigma$-strong* topology is quasi-complete.
\end{lemma}

\begin{proof}
First, we prove that $M$ equipped with its canonical locally convex structure which gives the strong topology is quasi-complete. Consider a bounded Cauchy net $(x_i)$ of elements of $M$ for this structure. Then, for any $h \in H$, $(x_i(h))$ is a Cauchy net of elements of $H$, hence converges. So $(x_i)$ is a net which converges for the strong topology and which is norm-bounded by the uniform boundedness principle \cite[Theorem 1.8.9]{KaR1}. It is well-known that this implies that $(x_i)$ converges strongly to some \textit{bounded} linear operator $x$. Since each $x_i$ belongs to $M$. We conclude that $x$ belongs to $M$, we conclude that the strong structure is quasi-complete.

Recall that a net $(x_i)$ is Cauchy in the strong* topology if and only if $(x_i)$ and $(x_i^*)$ are Cauchy in the strong operator topology and that $(x_i)$ converges to $x$ in the strong* topology if and only if $(x_i)$ and $(x_i^*)$ converge to $x$ and $x^*$ in the strong operator topology. Using this fact, it is not difficult to see that the canonical locally convex structure which gives the strong* topology is quasi-complete. Finally, recall that by \cite[Lemma 2.5]{Tak1} the $\sigma$-strong* topology coincide with the strong* topology on bounded subsets. So the proof is complete.
\end{proof}

\paragraph{Group von Neumann algebras} 
Let $G$ be a locally compact group equipped with a fixed left invariant Haar measure $\mu_G$. 
Let $\VN(G)$ be the von Neumann algebra of $G$ generated by the set $\big\{\lambda(g) : g \in \L^1(G)\big\}$. It is called the group von Neumann algebra of $G$ and is equal to the von Neumann algebra generated by the set $\{\lambda_s : s \in G\}$ where 
$\lambda_s  \co \L^2(G) \to \L^2(G), f \longmapsto (t \mapsto f(s^{-1}t))$ 
is the left translation by $s$.

\paragraph{Crossed products} We refer to \cite{Haa2}, \cite{Haa3}, \cite{Str1}, \cite{Sun} and \cite{Tak2}. Let $M$ be a von Neumann algebra acting on a Hilbert space $H$. Let $G$ be a locally compact group equipped with some left Haar measure $\mu_G$. Let $\alpha \co G \to M$ be a representation of $G$ on $M$ which is weak* continuous, i.e. for any $x \in M$ and any $y \in M_*$, the map $G \to M$, $s \mapsto \langle \alpha_{s}(x), y \rangle_{M,M_*}$ is continuous. For any $x \in M$, we define the operators $\pi(x)\co \L^2(G,H) \to \L^2(G,H)$ \cite[(2) page 263]{Str1} by
\begin{equation}
\label{}
\big(\pi(x) \xi \big)(s)
\ov{\mathrm{def}}{=}\alpha^{-1}_s(x) \xi(s),\quad
\quad  \xi \in \L^2(G, H), s \in G.
\end{equation}
These operators satisfy the following commutation relation \cite[(2) page 292]{Str1}:
\begin{equation}
\label{commutation-rules}
(\lambda_s \ot \Id_H) \pi(x) (\lambda_s \ot \Id_H)^*
= \pi(\alpha_{s}(x)),
\quad x \in M, s \in G.
\end{equation}
Recall that the crossed product of $M$ and $G$ with respect to $\alpha$ is the von Neumann algebra $M \rtimes_\alpha G=(\pi(M) \cup \{\lambda_s \ot \Id_{H}: s \in G\})''$ on the Hilbert space $\L^2(G,H)$ generated by the operators $\pi(x)$ and $\lambda_s \ot \Id_{H}$ where $x \in M$ and $s \in G$. By \cite[page 263]{Str1} or \cite[Proposition 2.5]{Dae1}, $\pi$ is a normal injective $*$-homomorphism from $M$ into $M \rtimes_\alpha G$ (hence $\sigma$-strong* continuous). 

We denote by $\mathcal{K}(G,M)$ the space of $\sigma$-strong* continuous function $f \co G \to M$, $s \mapsto f_s$ with compact support. If $f \in \mathcal{K}(G,M)$ then $f(G)$ is a $\sigma$-strong* compact subset of $M$, hence by \cite[Proposition 2.7 d)]{Osb1} a $\sigma$-strong* bounded subset of $M$. Hence it is a strong bounded subset and finally a norm-bounded subset of $M$ by the principle of uniform boundedness \cite[Theorem 1.8.9]{KaR1}. Note that by \cite[Proposition page 186]{Str1} and \cite[page 41]{Str1}, the bounded function $G \to M$, $s \mapsto \lambda_s \ot \Id_H$ is $\sigma$-strong* continuous and the norm-bounded function $s \mapsto \pi(f_s)$ is also $\sigma$-strong* continuous. Recall that the product of $M$ is $\sigma$-strong* continuous on bounded subsets by \cite[Proposition 2.4.5]{BrR1}. We infer\footnote{\thefootnote. In the book \cite{Str1}, the author considers weak* continuous functions, it is problematic since the product of $M$ is not weak* continuous even on bounded sets by \cite[Exercise 5.7.9]{KaR1} (indeed this latter fact is equivalent to the weak continuity of the product on bounded sets by \cite[Lemma 2.5]{Tak1}).} that the function $G \to M \rtimes_\beta G$, $s \mapsto \pi(f_s)(\lambda_s \ot \Id_H)$ is $\sigma$-strong* continuous with compact support. 
So, by Lemma \ref{Lemma-quasi-complete} and \cite[Corollary 2, III page 38]{Bou1} we can define the element $\int_G f_s \rtimes \lambda_s \d\mu_G(s)$ of the crossed product $M \rtimes_\alpha G$ by 
\begin{equation}
\label{def-integrale-crossed}
\int_G f_s \rtimes \lambda_s \d\mu_G(s)
\ov{\mathrm{def}}{=} \int_G \pi(f_s)(\lambda_s \ot \Id_H) \d\mu_G(s).
\end{equation}

%
%

If $f,g \in \mathcal{K}(G,M)$ then essentially by the proof of \cite[Lemma 2.3]{Haa2} and \cite[page 289]{Str1} the function $g*f \co G \to M$, $t \to \int_G g(s)f(s^{-1}t) \d\mu_G(s)$ is $\sigma$-strong* continuous and we have
\begin{equation}
\label{Def-produit-produit-croise}
\bigg( \int_G g_s \rtimes \lambda_s \d\mu_G(s)\bigg)\bigg( \int_G f_s \rtimes \lambda_s \d\mu_G(s)\bigg)
=\int_G (g*f)(s) \rtimes \lambda_s \d\mu_G(s).
\end{equation}

It is very easy to see that the space of elements $\int_G f_s \rtimes \lambda_s \d\mu_G(s)$ for $f \in \mathcal{K}(G,M)$ is a weak* dense subalgebra of $M \rtimes_\alpha G$.

The following is a particular case\footnote{\thefootnote. The function $u \co G \to \mathrm{U}(M)$ is a $\alpha$-1-cocycle.} of \cite[Proposition 3.5]{Tak3} and its proof, see also \cite[Theorem 1.7 (ii) p. 241]{Tak2}. Note that $M$ is abelian in the statement. With \cite[Proposition 2, III page 35]{Bou1}, the last part is an easy computation left to the reader.

\begin{prop}
\label{Prop-Takesaki}
Let $M$ be an abelian von Neumann algebra acting on a Hilbert space $H$ equipped with a continuous action $\alpha$ of a locally compact group $G$. Suppose that there exists a strongly continuous function $u \co G \to \mathrm{U}(M)$ such that
\begin{equation}
\label{equation-unitaries}
u(sr)
=u(s)\alpha_s(u(r)),\quad s,r \in G. 
\end{equation}
Then $V \co \L^2(G,H) \to \L^2(G,H)$, $\xi \mapsto (s \mapsto u(s^{-1})(\xi(s)))$ is a unitary and we have an $*$-isomorphism $\mathcal{U} \co M \rtimes_{\alpha} G \to  M \rtimes_{\alpha} G$, $x \mapsto VxV^*$ such that
\begin{equation}
\label{Ul-def}
\mathcal{U}(\lambda_s \ot \Id_H)
=\pi(u(s)^*)(\lambda_s \ot \Id_H)
\quad \text{and} \quad 
\mathcal{U}(\pi(x))
=\pi(x), \quad s \in G, x \in M.
\end{equation}
Moreover, for any $f \in \mathcal{K}(G,M)$, we have
\begin{equation}
\label{quoi}
\mathcal{U}\bigg(\int_G f_s \rtimes \lambda_s \d\mu_G(s) \bigg)
=\int_G u(s)^* f_s  \rtimes \lambda_s \d\mu_G(s).
\end{equation}
\end{prop}


%
Now, we suppose that $M$ is \textit{finite} and equipped with a normal finite faithful trace $\tau$. By \cite[Lemma 3.3]{Haa2} \cite[Theorem p. 301]{Str1} \cite[Theorem 1.17]{Tak2}, there exists a unique normal semifinite faithful weight $\varphi_{\rtimes}$ on $M \rtimes_{\beta} G$ which satisfies for any $f,g \in \mathcal{K}(G,M)$ the  fundamental ``noncommutative Plancherel formula''
\begin{equation}
\label{Plancherel-Non-com}
\varphi_{\rtimes}\bigg(\bigg( \int_G f_s \rtimes \lambda_s \d\mu_G(s)\bigg)^*\bigg( \int_G g_s \rtimes \lambda_s \d\mu_G(s)\bigg)\bigg)
=\int_G \tau(f_s^*g_s) \d\mu_G(s)
\end{equation}
and the relations $\sigma_t^{\varphi_{\rtimes}}(\pi(x))=\pi(x)$ where $x \in M$ and $t \in \R$
and 
\begin{equation*}
\sigma_t^{\varphi_{\rtimes}}(\lambda_s \ot \Id_H)
=\Delta_G^{\i t}(s)(\lambda_s \ot \Id_H)\pi([\D(\tau \circ \alpha_s):\D\tau]_t), \quad s \in G,t \in \R.
\end{equation*}
If $M=\C$, we recover the Plancherel weight $\varphi_G$ on $\VN(G)$. If each $\alpha_s \co M \to M$ is trace preserving, we obtain in particular 
\begin{equation*}
\label{modular-group-crossed-prime}
\sigma_t^{\varphi_{\rtimes}}(\lambda_s \ot \Id_H)
=\Delta_G^{\i t}(s)(\lambda_s \ot \Id_H), \quad s \in G,t \in \R.
\end{equation*}
Using \cite[Proposition 2, III page 35]{Bou1}, we deduce that
\begin{equation}
\label{modular-group-crossed}
\sigma_t^{\varphi_{\rtimes}}\bigg(\int_G f_s \rtimes \lambda_s \d\mu_G(s)\bigg)
=\int_G \Delta_G^{\i t}(s)f_s \rtimes \lambda_s \d\mu_G(s), \quad f \in \mathcal{K}(G,M), t \in \R.
\end{equation}

\paragraph{Fourier multipliers} Let $G$ be a locally compact group. We say that a weak* continuous operator $T \co \VN(G)) \to \VN(G)$ is a Fourier multiplier if there exists a continuous function $\phi \co G \to \C$ such that for any $s \in G$ we have $T(\lambda_s)=\phi(s)\lambda_s$. In this case $\phi$ is bounded and for any $f \in \L^1(G)$ the element $\int_{G} \phi(s) f(s) \lambda_s \d \mu_G(s)$ belongs to $\VN(G)$ and 
\begin{equation}
\label{equ-def-Fourier-mult}
T\bigg(\int_{G} f(s) \lambda_s \d \mu_G(s)\bigg) 
= \int_{G} \phi(s) f(s) \lambda_s \d \mu_G(s), \quad \text{i.e.} \quad T(\lambda(f)) =\lambda(\phi f).
\end{equation}
In this case, we let $T=M_\phi$ and we say that $\phi$ is the symbol of $T$. We refer to the book \cite{KaL1} and references therein for more information and to \cite{ArK1} for Fourier multipliers on noncommutative $\L^p$-spaces.

By \cite[Theorem 4.1]{HJX} (see also \cite[Proposition 3.13]{Dae1}), we have the following result. Note that the proof of \cite[Theorem 4.1]{HJX} does \textit{not} use the fact that $G$ is abelian. The second part is an obvious observation left to the reader.

\begin{lemma}
\label{Lemma-crossed}
Let $G$ be a locally compact group and $\alpha \co G \to \Aut(M)$ be a weak* continuous action on a von Neumann algebra $M$ equipped with a normal semifinite faithful weight. Let $T \co M \to M$ be a weight preserving $*$-automorphism such that $\E\alpha(s)=\alpha(s)\E$ for any $s \in G$.
\begin{enumerate}
\item There exists a weight preserving $*$-automorphism $T \rtimes \Id_{\VN(G)} \co M \rtimes_{\alpha} G \to M \rtimes_{\alpha} G$ 
such that for any $s \in G$ and any $x \in M$
\begin{equation}
\label{def-iso-croise}
\big(T \rtimes \Id_{\VN(G)}\big)(\pi(x))=\pi(T(x)),
\quad 
\big(T \rtimes \Id_{\VN(G)}\big)(\lambda_s \ot \Id_{H})=\lambda_s \ot \Id_{H}.
\end{equation}
\item For any function $f \in \mathcal{K}(G,M)$, we have 
\begin{equation}
\label{crossed-2}
\big(T \rtimes \Id_{\VN(G)}\big)\bigg(\int_G f_s \rtimes \lambda_s \d\mu_G(s) \bigg)
	=\int_G T(f_s) \rtimes \lambda_s \d\mu_G(s).
\end{equation}
\end{enumerate}
\end{lemma}

\paragraph{Semigroups of Fourier multipliers} Consider a locally compact group $G$ with identity element $e$. Let $(T_t)_{t \geq 0}$ be a weak* continuous semigroup of selfadjoint unital completely positive Fourier multipliers. For any $t \geq 0$, let $\phi_t \co G \to \C$ be the continuous symbol of $T_t$. Since each $T_t$ is selfadjoint, $\phi_t$ is real-valued. By \cite[Proposition 4.2]{DCH}, the function $\phi_t$ is of positive type. Moreover, for any $t,t' \geq 0$ and any $s \in G$, the relation $T_tT_{t'}(\lambda_s)=T_{t+t'}(\lambda_s)$ gives $\phi_t(s)\phi_{t'}(s)=\phi_{t+t'}(s)$. Furthermore, for any $s \in G$, any $x \in \VN(G)_*$ and any $t \geq 0$, we have 
$$
\phi_t(s)\langle \lambda_s, x\rangle_{\VN(G),\VN(G)_*}
=\langle T_t(\lambda_s), x\rangle_{\VN(G),\VN(G)_*} 
\xra[t \to 0]{} \langle\lambda_s, x\rangle_{\VN(G),\VN(G)_*}.
$$
We infer that $\lim_{t \to 0} \phi_t(s)=1$. Consequently, there exists a uniquely determined real number $\psi(s)$ such that
$$
\phi_t(s)
=\e^{-t\psi(s)}, \quad t \geq 0, s \in G.
$$ 
Since each $T_t$ is unital, we have $\psi(e)=0$. It is easy to check that $\psi \co G \to \R$ is continuous. By Schoenberg's theorem \cite[Corollary C.4.19]{BHV}, we deduce that the function $\psi$ is conditionally of negative type.  

By \cite[Proposition 2.10.2]{BHV}, there exist a real Hilbert space $H$ and an affine isometric action $\beta \co G \to \mathrm{Isom}(H)$, $s \mapsto \beta_s$ such that the linear span of $\{\beta_s(0) : s \in G\}$ is dense in $H$ and such that
\begin{equation}
\label{psiG=} 
\psi(s)
=\norm{\beta_s(0)}_H^2, \quad s \in G.       
\end{equation}  
By \cite[page 75]{BHV}, we can consider the associated orthogonal representation $\pi \co G \to \B(H)$, $s \mapsto \pi_s$ of $G$ on $H$. By \cite[Lemma 2.2.1]{BHV}, there exists a 1-cocycle $b_\psi \co G \to H$ with respect to $\pi$, i.e. a continuous function satisfying the cocycle law 
\begin{equation}
\label{Cocycle-law}
\pi_s(b_\psi(r))
=b_\psi(sr)-b_\psi(s),
\quad \text{i.e.} \quad 
b_\psi(sr)
=b_\psi(s)+\pi_s(b_\psi(r)), \quad s,r \in G,
\end{equation}
such that 
\begin{equation}
\label{lien-affine-cocycle}
\beta_s(h)
=\pi_s(h)+b_\psi(s), \quad h \in H.
\end{equation}
For any $s \in G$, we deduce that 
\begin{equation}
\label{liens-psi-bpsi}
\psi(s)
\ov{\eqref{psiG=}}{=} \norm{\beta_s(0)}_H^2
\ov{\eqref{lien-affine-cocycle}}{=} \norm{b_\psi(s)}_{H}^2.
\end{equation}
\section{Dilations of semigroups of Fourier multipliers}
\label{sec:dilation-continuous}

The following theorem is our main result.

\begin{thm}
\label{Th-dilation-continuous}
Let $G$ be a locally compact group. Let $(T_t)_{t \geq 0}$ be a weak* continuous semigroup of selfadjoint unital completely positive Fourier multipliers on the group von Neumann algebra $\VN(G)$. Then, there exist a von Neumann algebra $M$ equipped with a semifinite normal faithful weight $\varphi_M$, a weak* continuous group $(U_t)_{t \in \R}$ of weight preserving $*$-automorphisms of $M$, a unital weight preserving injective normal $*$-homomorphism $J \co \VN(G) \to M$ such that 
\begin{equation}
\label{commute-modular-group}
\sigma_t^{\varphi_M} \circ J 
=J \circ \sigma_t^{\varphi_G}
\end{equation}
for any $t \in \R$ and satisfying
\begin{equation}
\label{Tt-dilation}
T_t
=\E U_tJ
\end{equation}
for any $t \geq 0$, where $\E \co M \to \VN(G)$ is the canonical faithful normal weight preserving conditional expectation associated with $J$. Moreover, we have the following properties.
\begin{enumerate}
	\item If $G$ is discrete then $\varphi_M$ is a normal finite faithful trace.
	\item If $G$ is unimodular then $\varphi_M$ is a normal semifinite faithful trace.
	\item If $G$ is amenable then the von Neumann algebra $M$ is injective.
\end{enumerate}
\end{thm}

\begin{proof}
Here we use the continuous function $\psi \co G \to \R$, the 1-cocycle $b_\psi \co G \to H$ and the orthogonal representation $\pi \co G \to \B(H)$, $s \mapsto \pi_s$ of $G$ on $H$ of Section \ref{sec:preliminaries}. Let $\W \co \L^2_\R(\R,H) \to \L^0(\Omega)$ be an $\L^2_\R(\R,H)$-isonormal process on a probability space $(\Omega,\mu)$, see again Section \ref{sec:preliminaries}. For any $s \in G$, we will use the second quantization $\alpha_s \ov{\mathrm{def}}{=} \Gamma_\infty(\Id \ot \pi_s) \co \L^\infty(\Omega) \to \L^\infty(\Omega)$ which is integral preserving. In particular, if $r,s \in G$ and if $t \in \R$, we have
\begin{equation}
	\label{Action-on-bpsi}
	\alpha_s\big(\e^{-\sqrt{2}\i \W_t(b_\psi(r))}\big)
	=\Gamma_\infty(\Id \ot \pi_s)\big(\e^{-\sqrt{2}\i \W_t(b_\psi(r))}\big)
	\ov{\eqref{SQ1}}{=} \e^{-\sqrt{2}\i \W_t(\pi_s(b_\psi(r)))}.
\end{equation}
Since $\pi$ is strongly continuous, by Lemma \ref{Lemma-semigroup-continuous}, we obtain a continuous action $\alpha \co G \to \Aut(\L^\infty(\Omega))$. So we can consider the crossed product $M \ov{\mathrm{def}}{=} \L^\infty(\Omega) \rtimes_{\alpha} G$ equipped with its canonical normal semifinite faithful weight $\varphi_M \ov{\mathrm{def}}{=}\varphi_\rtimes$. We denote by $J \co \VN(G) \to \L^\infty(\Omega) \rtimes_{\alpha} G$, $\lambda_s \mapsto 1 \rtimes \lambda_s$ the canonical unital normal injective $*$-homomorphism. Using \cite[Proposition 2, III page 35]{Bou1}, for any $f \in \mathcal{K}(G)$, we see that
\begin{equation}
\label{Def-de-J-prem}
J\bigg(\int_G f(s) \lambda_s \d\mu_G(s)\bigg)
=\int_G f(s)1 \rtimes \lambda_s \d\mu_G(s).
\end{equation}

\begin{lemma}
We have $\varphi_\rtimes  \circ J=\varphi_G$.
\end{lemma} 

\begin{proof}
We will use \cite[Theorem 6.2]{Str1} with the weights $\varphi_G$ and $\varphi_\rtimes  \circ J$. Note\footnote{\thefootnote. Recall that $\lambda(f)=\int_{G} f(s) \lambda_s \d \mu_G(s)$.} that $\lambda(\mathcal{K}(G))$ is a $*$-subalgebra which is $\sigma^{\varphi_G}$-invariant, weak* dense in $\VN(G)$ and included in $\mathfrak{n}_{\varphi_G}$. 

For any $f \in \mathcal{K}(G)$, we have $f(\cdot)1 \in \mathcal{K}(G,\L^\infty(\Omega))$. So $J(\lambda(f)) \ov{\eqref{Def-de-J-prem}}{=}\int_G f(s)1 \rtimes \lambda_s \d\mu_G(s)$ belongs to $\mathfrak{n}_{\varphi_\rtimes}$. Hence $\varphi_\rtimes \circ J\big(\lambda(f)^*\lambda(f)\big)=\varphi_\rtimes\big(J(\lambda(f))^*J(\lambda(f))\big)<\infty$. Consequently, we have $\lambda(f) \in \mathfrak{n}_{\varphi_\rtimes \circ J}$. We infer that $\mathfrak{n}_{\varphi_\rtimes \circ J}$ is weak* dense in $\VN(G)$. By \cite[page 19]{Str1}, we conclude that the weight $\varphi_\rtimes \circ J$ is semifinite. It is obvious that $\varphi_\rtimes \circ J$ is normal. For any $t \in \R$, note that
$$
\varphi_\rtimes \circ J \circ \sigma_t^{\varphi_G}
\ov{\eqref{commute-modular-group}}{=}\varphi_\rtimes \circ \sigma_t^{\varphi_\rtimes} \circ J 
=\varphi_\rtimes \circ J.
$$ 
So the weights $\varphi_\rtimes \circ J$ and $\varphi_G$ commutes by \cite[pages 67-68]{Str1}. Now, if $f \in \mathcal{K}(G)$, we have
\begin{align*}
\MoveEqLeft
\varphi_\rtimes \circ J\big(\lambda(f)^*\lambda(f)\big)           
		=\varphi_\rtimes\big(J(\lambda(f))^*J(\lambda(f))\big)   \\
		&\ov{\eqref{Def-de-J-prem}}{=}\varphi_\rtimes\bigg(\bigg(\int_G f(s)1 \rtimes \lambda_s \d\mu_G(s)\bigg)^*\int_G f(s)1 \rtimes \lambda_s \d\mu_G(s)\bigg) \\
		&\ov{\eqref{Plancherel-Non-com}}{=}\int_G \bigg(\int_\Omega |f(s)|^21 \d \mu\bigg)\d\mu_G(s)
		=\int_G |f(s)|^2 \d\mu_G(s) 
		\ov{\eqref{Plancherel-Non-com}}{=} \varphi_G\big(\lambda(f)^*\lambda(f)\big).
\end{align*}  
We conclude with \cite[Theorem 6.2]{Str1} that $\varphi_\rtimes  \circ J=\varphi_G$.
\end{proof}

For any $t \in \R$, we consider the function $u_t \co G \to \mathrm{U}(\L^\infty(\Omega))$, $s \mapsto \e^{-\sqrt{2}\i \W_t(b_\psi(s))}$ where $\W_t(b_\psi(s)) \ov{\mathrm{def}}{=} \W(-1_{[t,0]} \ot b_\psi(s))$ if $t<0$. The map $b_\psi \co G \to H$ is continuous. By the point 1 of Lemma \ref{Lemma-semigroup-continuous}, the map $\L^2_\R(\R,H) \to \L^\infty(\Omega)$, $g \mapsto \e^{\i\W(g)}$ is continuous if $\L^\infty(\Omega)$ is equipped with the weak* topology, hence with the weak operator topology when we consider that $\L^\infty(\Omega)$ acts on $\L^2(\Omega)$. By composition, the function $u_t$ is continuous. Recall that by \cite[Exercice 5.7.5]{KaR1} or \cite[page 41]{Str1} the weak operator topology and the strong operator topology coincide on the unitary group $\mathrm{U}(\L^\infty(\Omega))$. So the function $u_t$ is continuous if $\mathrm{U}(\L^\infty(\Omega))$ is equipped with the strong operator topology. For any $t \in \R$ and any $r,s \in G$, note that  
\begin{align*}
\MoveEqLeft
u_t(sr)
=\e^{-\sqrt{2}\i  \W_t(b_\psi(sr))}
\ov{\eqref{Cocycle-law}}{=}
\e^{-\sqrt{2}\i  \W_t(b_\psi(s))}\e^{-\sqrt{2}\i  \W_t(\pi_s(b_\psi(r)))}\\
&\ov{\eqref{Action-on-bpsi}}{=}u_t(s)\alpha_s\big(\e^{-\sqrt{2}\i \W_t(b_\psi(r))}\big)
=u_t(s)\alpha_s(u_t(r)).            
\end{align*}
Hence \eqref{equation-unitaries} is satisfied. By Proposition \ref{Prop-Takesaki}, for any $t \in \R$, we have a unitary $V_t \co \L^2(G,\L^2(\Omega)) \to \L^2(G,\L^2(\Omega))$, $\xi \mapsto (s \mapsto u_t(s^{-1})\xi(s))$ and a $*$-isomorphism $\mathcal{U}_t \co \L^\infty(\Omega) \rtimes_{\alpha} G \to \L^\infty(\Omega) \rtimes_{\alpha} G$, $x \mapsto V_txV_t^*$ such that
\begin{equation}
\label{Ul-bis}
\mathcal{U}_t(\lambda_s \ot \Id_H)
\ov{\eqref{Ul-def}}{=} \pi\big(\e^{\sqrt{2}\i \W_t(b_\psi(s))}\big)(\lambda_s \ot \Id_H)
\quad \text{and} \quad 
\mathcal{U}_t(\pi(x))
\ov{\eqref{Ul-def}}{=} \pi(x), \quad s \in G, x \in \L^\infty(\Omega).
\end{equation}
and finally for any $f \in \mathcal{K}(G,M)$
\begin{equation}
\label{Def-Ut}
\mathcal{U}_t\bigg(\int_G f_s \rtimes \lambda_s \d\mu_G(s)\bigg)
\ov{\eqref{quoi}}{=} \int_G \e^{\sqrt{2}\i \W_t(b_\psi(s))}f_s  \rtimes \lambda_s \d\mu_G(s).
\end{equation}
\begin{lemma}
Each $*$-automorphism $\mathcal{U}_t$ preserves the weight $\varphi_\rtimes$.
\end{lemma} 

\begin{proof}
We will use \cite[Theorem 6.2]{Str1} with the weights $\varphi_\rtimes$ and $\varphi_\rtimes \circ \mathcal{U}_t$. Note that the space of elements $\int_G f_s \rtimes \lambda_s \d\mu_G(s)$ for $f \in \mathcal{K}(G,\L^\infty(\Omega))$ is a $*$-subalgebra which is $\sigma^{\varphi_\rtimes}$-invariant, weak* dense in $\L^\infty(\Omega) \rtimes_{\alpha} G$ and included in $\mathfrak{n}_{\varphi_\rtimes}$. The formulas \eqref{modular-group-crossed} and \eqref{Def-Ut} show that each $\mathcal{U}_t$ and $\sigma_t^{\varphi_\rtimes}$ commute. So, we have
$$
\varphi_\rtimes \circ \mathcal{U}_t \circ \sigma_t^{\varphi_\rtimes}
=\varphi_\rtimes \circ \sigma_t^{\varphi_\rtimes} \circ U_t
=\varphi_\rtimes \circ \mathcal{U}_t.
$$
So the weights $\varphi_\rtimes \circ \mathcal{U}_t$ and $\varphi_\rtimes$ commutes by \cite[pages 67-68]{Str1}. It is (really) easy to check that the weight $\varphi_\rtimes \circ \mathcal{U}_t$ is normal and semifinite. If $f \in \mathcal{K}(G,\L^\infty(\Omega))$, we have
\begin{align*}
\MoveEqLeft
\varphi_\rtimes  \circ \mathcal{U}_t\bigg(\bigg(\int_G f_s \rtimes \lambda_s \d\mu_G(s)\bigg)^*\bigg(\int_G f_s \rtimes \lambda_s \d\mu_G(s)\bigg)\bigg) \\          
		&=\varphi_\rtimes\Bigg(\bigg(\mathcal{U}_t\bigg(\int_G f_s \rtimes \lambda_s \d\mu_G(s)\bigg)\bigg)^*\mathcal{U}_t\bigg(\int_G f_s \rtimes \lambda_s \d\mu_G(s)\bigg)\Bigg) \\
		&\ov{\eqref{Def-Ut}}{=}\varphi_\rtimes\Bigg(\bigg(\int_G \e^{\sqrt{2}\i \W_t(b_\psi(s))}f_s  \rtimes \lambda_s \d\mu_G(s)\bigg)^*\bigg(\int_G \e^{\sqrt{2}\i \W_t(b_\psi(s))}f_s  \rtimes \lambda_s \d\mu_G(s)\bigg)\Bigg)\\
		&\ov{\eqref{Plancherel-Non-com}}{=} \int_G \int_\Omega\e^{-\sqrt{2}\i \W_t(b_\psi(s))}f_s^*\e^{\sqrt{2}\i \W_t(b_\psi(s))}f_s \d \mu \d\mu_G(s)
		=\int_G \int_\Omega f_s^*f_s\d \mu \d\mu_G(s)\\
		&\ov{\eqref{Plancherel-Non-com}}{=} \varphi_\rtimes\bigg(\bigg(\int_G f_s \rtimes \lambda_s \d\mu_G(s)\bigg)^*\bigg(\int_G f_s \rtimes \lambda_s \d\mu_G(s)\bigg)\bigg).
\end{align*} 
 We conclude with \cite[Theorem 6.2]{Str1} that $\varphi_\rtimes \circ \mathcal{U}_t=\varphi_\rtimes$ for any $t \geq 0$.
\end{proof}
\newpage

For any $t \in \R$, we introduce the right shift $\mathcal{S}_t \co \L^2_\R(\R) \to \L^2_\R(\R)$. If $t,u \geq 0$, we have 
\begin{align*}
\MoveEqLeft
\alpha_s\Gamma_\infty(\mathcal{S}_t \ot \Id_H)(\e^{\sqrt{2}\i  \W_u(h)})            
=\alpha_s\big( \e^{\sqrt{2}\i  \W_{]t,t+u]}(h)} \big) 
=\e^{\sqrt{2}\i  \W_{]t,t+u]}(\pi_s(h))}
\end{align*} 
and
\begin{align*}
\MoveEqLeft
\Gamma_\infty(\mathcal{S}_t \ot \Id_H)\alpha_s(\e^{\sqrt{2}\i  \W_u(h)})            
\ov{\eqref{Action-on-bpsi}}{=} \Gamma_\infty(\mathcal{S}_t \ot \Id_H)\e^{\sqrt{2}\i \W_u(\pi_s(h))}
=\e^{\sqrt{2}\i \W(1_{]t,t+u]} \ot \pi_s(h))}.
\end{align*} 
The remaining cases are left to the reader. Hence by density $\Gamma_\infty(\mathcal{S}_t \ot \Id_H)\alpha_s=\alpha_s \Gamma_\infty(\mathcal{S}_t \ot \Id_H)$ for any $s \in G$ and any $u \geq 0$. 

With Lemma \ref{Lemma-crossed}, we consider the operator $S_t \ov{\mathrm{def}}{=} \Gamma_\infty(\mathcal{S}_t \ot \Id_H) \rtimes \Id_{\VN(G)} \co \L^\infty(\Omega) \times_\alpha G \to \L^\infty(\Omega) \times_\alpha G$. Essentially, by the part 2 of Lemma \ref{Lemma-semigroup-continuous}, $(S_t)_{t \in \R}$ is a weak* continuous semigroup of $*$-automorphisms which preserves the weight. We have
\begin{equation}
\label{def-iso-croise-2}
S_t(\pi(\e^{\i \W(1_{[u,u'[} \ot h )}))
\ov{\eqref{def-iso-croise}}{=} \pi(\e^{\i \W(1_{[u+t,u'+t[} \ot h )})),
\quad 
S_t(\lambda_s \ot \Id_{H})
\ov{\eqref{def-iso-croise}}{=} \lambda_s \ot \Id_{H}
\end{equation}
and
\begin{equation}
\label{crossed-3}
S_t\bigg(\int_G f_s \rtimes \lambda_s \d\mu_G(s) \bigg)
\ov{\eqref{crossed-2}}{=}	\int_G T(f_s) \rtimes \lambda_s \d\mu_G(s).
\end{equation}
For any $t \in \R$, we define the $*$-automorphism $U_t \ov{\mathrm{def}}{=} \mathcal{U}_t S_t$. We will prove that $(U_t)_{t \in \R}$ is a group of operators. 
On the one hand, for any $t,t' \in \R$, we have
\begin{align*}
\MoveEqLeft
U_{t'}U_t\big(\pi(\e^{\i \W(1_{[u,u']} \ot h)}) \big)             
=\mathcal{U}_{t'} S_{t'} \mathcal{U}_t S_{t}\big( \pi(\e^{\i \W(1_{[u,u']} \ot h)})\big) 
\ov{\eqref{def-iso-croise-2}}{=} \mathcal{U}_{t'} S_{t'} \mathcal{U}_t \big( \pi(\e^{\i \W(1_{[u+t,u'+t]} \ot h)})\big) \\
&\ov{\eqref{Ul-bis}}{=} \mathcal{U}_{t'} S_{t'}\big( \pi(\e^{\i \W(1_{[u+t,u'+t]} \ot h)})\big)
\ov{\eqref{def-iso-croise-2}}{=} \mathcal{U}_{t'} \big( \pi(\e^{\i \W(1_{[u+t+t',u'+t+t']} \ot h)})\big) \\
&\ov{\eqref{def-iso-croise-2}}{=} \pi(\e^{\i \W(1_{[u+t+t',u'+t+t']} \ot h)} 
=\mathcal{U}_{t'+t} S_{t'+t}\big(\pi(\e^{\i \W(1_{[u,u']} \ot h)}) \big)
=U_{t'+t}\big(\pi(\e^{\i \W(1_{[u,u']} \ot h)}) \big).
\end{align*}
On the other hand, for any $t,t' \geq 0$, we have
\begin{align*}
\MoveEqLeft
U_{t'}U_t(\lambda_s \ot \Id_H)          
=\mathcal{U}_{t'} S_{t'} \mathcal{U}_t(\lambda_s \ot \Id_H) 
\ov{\eqref{Ul-bis}}{=} \mathcal{U}_{t'} S_{t'}\Big( \pi\big(\e^{\sqrt{2}\i \W_t(b_\psi(s))}\big)(\lambda_s \ot \Id_H)\Big) \\
&\ov{\eqref{def-iso-croise-2}}{=} \mathcal{U}_{t'} \Big(\pi \big(\e^{\sqrt{2}\i \W_{[t',t+t']}(b_\psi(s))}\big)(\lambda_s \ot \Id_H)\Big) \\
&\ov{\eqref{Ul-bis}}{=} \pi \Big(\e^{\sqrt{2}\i \W_{[t',t+t']}(b_\psi(s))} \e^{\sqrt{2}\i \W_{t'}(b_\psi(s))}\Big)(\lambda_s \ot \Id_H)\\
&\ov{\eqref{Ul-bis}}{=} \pi\big(\e^{\sqrt{2}\i \W_{t+t'}(b_\psi(s))}\big)(\lambda_s \ot \Id_H)
=\mathcal{U}_{t+t'}(\lambda_s \ot \Id_H)  \\
&\ov{\eqref{def-iso-croise-2}}{=} \mathcal{U}_{t+t'} S_{t+t'}(\lambda_s \ot \Id_H)
=U_{t+t'}(\lambda_s \ot \Id_H).
\end{align*} 
Similarly for any $t \leq 0$ and any $t' \geq 0$ such that $t' \leq -t$ 
\begin{align*}
\MoveEqLeft
U_{t'}U_t (\lambda_s \ot \Id_H)          
=\mathcal{U}_{t'} S_{t'} \mathcal{U}_t (\lambda_s \ot \Id_H)  \\
&\ov{\eqref{Ul-bis}}{=} \mathcal{U}_{t'} S_{t'}\Big( \pi\big(\e^{\sqrt{2}\i \W(-1_{[t,0]} \ot b_\psi(s))}\big)(\lambda_s \ot \Id_H)\Big) \\
&\ov{\eqref{def-iso-croise-2}}{=} \mathcal{U}_{t'} \Big(\pi\big(\e^{\sqrt{2}\i \W(-1_{[t'+t,t']} \ot b_\psi(s))}\big)(\lambda_s \ot \Id_H) \Big) \\
&\ov{\eqref{Ul-bis}}{=}  \pi\big(\e^{\sqrt{2}\i \W(-1_{[t'+t,t']} \ot b_\psi(s))} \e^{\sqrt{2}\i \W_{t'}(b_\psi(s))} \big)(\lambda_s \ot \Id_H) \\
&= \pi\big(\e^{\sqrt{2}\i \W(-1_{[t'+t,0]} \ot b_\psi(s))}\big)(\lambda_s \ot \Id_H) \quad \text{since } t'+t \leq 0 \leq t' \\
&\ov{\eqref{Ul-bis}}{=} \mathcal{U}_{t+t'}(\lambda_s \ot \Id_H) 
\ov{\eqref{def-iso-croise-2}}{=} \mathcal{U}_{t+t'}S_{t+t'}(\lambda_s \ot \Id_H)
=U_{t+t'}(\lambda_s \ot \Id_H) .
\end{align*} 
The remaining cases are left to the reader. We conclude that $(U_t)_{t \in \R}$ is a weak* continuous group of $*$-automorphisms (consider the preadjoints maps of the $U_t$'s and use \cite[Lemme B.15]{EnN1} to prove the weak* continuity of the group).

With \cite[Theorem 10.1]{Str1}, we introduce the canonical faithful normal weight preserving conditional expectation $\E \co \L^\infty(\Omega) \rtimes_{\alpha} G \to \VN(G)$.
Using \cite[Theorem 7.5]{AcC}, it is entirely left to the reader to check that the conditional expectation is given by
\begin{equation}
\label{Esperance}
\E\bigg(\int_G g_s \rtimes \lambda_s \d\mu_G(s)\bigg)	
\ov{}{=}\int_G \mathrm{E}(g_s) \lambda_s \d\mu_G(s).
\end{equation}
For any $t \geq 0$, we have
\begin{align*}
\MoveEqLeft
\E U_tJ\bigg(\int_G f(s) \lambda_s \d\mu_G(s) \bigg) 
\ov{\eqref{Def-de-J-prem}}{=}\E \mathcal{U}_t S_t\bigg(\int_G f(s)1 \rtimes \lambda(s) \d\mu_G(s)\bigg) \\
&\ov{\eqref{crossed-3}}{=} \E \mathcal{U}_t \bigg(\int_G f(s)1 \rtimes \lambda(s) \d\mu_G(s)\\
&\ov{\eqref{Def-Ut}}{=}\E\bigg(\int_G \e^{\sqrt{2}\i \W_t(b_\psi(s))} f(s) \rtimes \lambda_s \d\mu_G(s)\bigg)\\ 
&\ov{\eqref{Esperance}}{=} \int_G \mathrm{E}\big(\e^{\sqrt{2}\i \W_t(b_\psi(s))} f(s)\big) \lambda_s \d\mu_G(s)
=\int_G \mathrm{E}\big(\e^{\sqrt{2}\i \W_t(b_\psi(s))}\big) f(s)\lambda_s \d\mu_G(s)\\
&\ov{\eqref{Esperance-exponentielle-complexe-2}}{=} \int_G \e^{-\frac{t}{2} 2\norm{b_\psi(s)}_{H}^2} f(s)\lambda_s \d\mu_G(s)
=\int_G\e^{-t\norm{b_\psi(s)}_{H}^2} f(s)\lambda_s \d\mu_G(s)\\
&\ov{\eqref{liens-psi-bpsi}}{=}\int_G \e^{-t\psi(s)} f(s)\lambda_s \d\mu_G(s)
\ov{\eqref{equ-def-Fourier-mult}}{=} T_t\bigg(\int_G f(s) \lambda_s \d\mu_G(s)\bigg).
\end{align*}
Thus by density, for any $t \geq 0$, we obtain \eqref{Tt-dilation}. Now, we prove the last assertions. Note each $\alpha_s \co \L^\infty(\Omega) \to \L^\infty(\Omega)$ is integral-preserving. The first is well-known, e. g. \cite[Corollary 7.11.8]{Ped1}. The second is folklore. The third is \cite[Proposition page 301]{Ana1}. 
\end{proof}

\begin{remark} \normalfont
In the case of the Heat semigroup $(\mathcal{H}_{t})_{t \geq 0}$ on $\L^\infty(\R^n)$, we can take $H=\R^n$, $b_\psi=\Id_{\R^n}$ and the action $\alpha$ is trivial. Here, we does not need a noncommutative von Neumann algebra for the dilation since $M$ is a tensor product of commutative von Neumann algebras. This observation is a complement to \cite[Secton 4]{Arh4}.

\end{remark}


\section{Functional calculus}
\label{sec:application}


We start with a little background on sectoriality and $\H^\infty$ functional calculus. We refer to \cite{Haa}, \cite{KW}, \cite{JMX}, \cite{HvNVW2} and \cite{Arh2} for details and complements. Let $X$ be a Banach space. A closed densely defined linear operator $A \co \dom A \subset X \to X$ is called sectorial of type $\omega$ if its spectrum $\sigma(A)$ is included in the closed sector $\overline{\Sigma_\omega}$ where $\Sigma_\omega \ov{\mathrm{def}}{=} \{z \in \C^*: |\arg z|< \omega\}$, and for any angle $\omega<\theta < \pi$, there is a positive constant $K_\theta$ such that
\begin{equation*}\label{Sector}
 \bnorm{(\lambda-A)^{-1}}_{X\to X}\leq
\frac{K_\theta}{\vert  \lambda \vert},\qquad \lambda
\in\mathbb{C}-\overline{\Sigma_\theta}.
\end{equation*}
If $-A$ is the negative generator of a bounded strongly continuous semigroup on a $X$ then $A$ is sectorial of type $\frac{\pi}{2}$. By \cite[Example 10.1.3]{HvNVW2}, sectorial operators of type $<\frac{\pi}{2}$ coincide with negative generators of bounded analytic semigroups.

For any $0<\theta<\pi$, let $\H^{\infty}(\Sigma_\theta)$ be the algebra of all bounded analytic functions $f \co  \Sigma_\theta\to \C$, equipped with the supremum norm $\norm{f}_{\H^{\infty}(\Sigma_\theta)}=\,\sup\bigl\{\vert f(z)\vert \, :\, z\in \Sigma_\theta\bigr\}$. Let $\H^{\infty}_{0}(\Sigma_\theta)\subset \H^{\infty}(\Sigma_\theta)$ be the subalgebra of bounded analytic functions $f \co \Sigma_\theta\to \C$ for which there exist $s,c>0$ such that $\vert f(z)\vert\leq \frac{c \vert z \vert^s}{(1+\vert z \vert)^{2s}}$ for any $z \in \Sigma_\theta$. 

Given a sectorial operator $A$ of type $0< \omega < \pi$, a bigger angle $\omega<\theta<\pi$, and a function $f\in \H^{\infty}_{0}(\Sigma_\theta)$, one may define a bounded operator $f(A)$ by means of a Cauchy integral (see e.g. \cite[Section 2.3]{Haa} or \cite[Section 9]{KW}). The resulting mapping $\H^{\infty}_{0}(\Sigma_\theta) \to \B(X)$ taking $f$ to $f(A)$ is an algebra homomorphism. By definition, $A$ has a bounded $\H^{\infty}(\Sigma_\theta)$ functional calculus provided that this homomorphism is bounded, that is if there exists a positive constant $C$ such that $\bnorm{f(A)}_{X \to X} \leq C\norm{f}_{\H^{\infty}(\Sigma_\theta)}$ for any $f \in \H^{\infty}_{0}(\Sigma_\theta)$. In the case when $A$ has a dense range, the latter boundedness condition allows a natural extension of $f\mapsto f(A)$ to the full algebra $\H^{\infty}(\Sigma_\theta)$.


Using the connection between the existence of dilations in UMD spaces and $\H^\infty$ functional calculus together with the well-known angle reduction principle of Kalton-Weis relying on R-sectoriality, Theorem \ref{Th-dilation-continuous} allows us to obtain the following result (see \cite[Proposition 3.12]{JMX}, \cite[Proposition 5.8]{JMX} and \cite[Corollary 10.9]{KW}). Here, we confine ourselves to unimodular groups by simplicity. Non-unimodular extensions and vector-valued versions of this result are left to the reader. Other applications will be given in subsequent papers, e.g. \cite{ArK2}. 

\begin{thm}
\label{Th-funct-calculus}
Let $G$ be unimodular locally compact group. Let $(T_t)_{t \geq 0}$ be a weak* continuous semigroup of selfadjoint unital completely positive Fourier multipliers on $\VN(G)$. Suppose $1<p<\infty$. We let $-A_p$ be the generator of the induced strongly continuous semigroup $(T_{t,p})_{t \geq 0}$ on the Banach space $\L^p(\VN(G))$. Then for any $\theta>\pi|\frac{1}{p}-\frac{1}{2}|$, the operator $A_p$ has a completely bounded $\H^{\infty}(\Sigma_\theta)$ functional calculus.
\end{thm}

\vspace{0.3cm}

\textbf{Acknowledgements}.
The author acknowledges support by the grant ANR-18-CE40-0021 (project HASCON) of the French National Research Agency ANR. Finally, I would like to thank the referee for useful remarks.



\vspace{0.5cm}

\vspace{0.2cm}
\footnotesize{
\noindent C\'edric Arhancet\\ 
\noindent13 rue Didier Daurat, 81000 Albi, France\\
URL: \href{http://sites.google.com/site/cedricarhancet}{https://sites.google.com/site/cedricarhancet}\\
cedric.arhancet@protonmail.com\\

\end{document}